\pdfoutput=1
\documentclass[10pt]{article}
\font\smallit=cmti10

\usepackage{amssymb,latexsym,amsmath,epsfig,amsthm,graphicx}
\graphicspath{ {images/} }

\makeatletter

\renewcommand\section{\@startsection {section}{1}{\z@}
{-30pt \@plus -1ex \@minus -.2ex}
{2.3ex \@plus.2ex}
{\normalfont\normalsize\bfseries\boldmath}}

\renewcommand\subsection{\@startsection{subsection}{2}{\z@}
{-3.25ex\@plus -1ex \@minus -.2ex}
{1.5ex \@plus .2ex}
{\normalfont\normalsize\bfseries\boldmath}}

\renewcommand{\@seccntformat}[1]{\csname the#1\endcsname. }

\makeatother

\newtheorem{theorem}{Theorem}

\newtheorem{proposition}{Proposition}

\begin{document}

\begin{center}
{\bf \large Strategy Stealing in Triangle Avoidance Games}
\vskip 20pt
{\bf Alexandru Malekshahian}\\
{\smallit Peterhouse, Trumpington Street, Cambridge CB2 1RD, UK}\\
{\tt am2580@cam.ac.uk}
\end{center}
\vskip 20pt

\vskip 30pt

\centerline{\bf Abstract}
\noindent
In the game of $n-Sim$, two players take it in turn to claim unclaimed edges from a complete graph on $n$ vertices, with the first person to create a triangle in his own edges being the loser. We present some strategy-stealing arguments that show that certain positions are wins for the second player.
These are among the only strategy-stealing arguments that are known for mis\`ere games.

\thispagestyle{empty} 
\baselineskip=12.875pt 
\vskip 30pt

\section*{\normalsize 1. Introduction}

The game of $n-Sim$ is defined as follows: given $n$ vertices, two players, PI and PII, take turns colouring one previously uncoloured edge, with PI going first. The players use different colours, and the first player that is forced to draw a triangle with all edges of the same colour loses. For $n \geq 6$, the game can never end in a draw, and we wish to determine which player has the winning strategy.

The first reference to this game that we are aware of is in Simmons [3], with Slany [4] using computer analysis to prove that PII has a winning strategy for $n=6$. $n-Sim$ is an example of an {\it avoidance}, or {\it mis\`ere} game. 

For achievement games, a simple strategy-stealing argument shows that, with perfect play, PII can never win. However, the problem of adapting this argument to the case of mis\`ere games is still open. For a more general discussion of  avoidance games, see, e.g., Beck [1] or Johnson, Leader and Walters [2].

In this paper, we make use of strategy stealing to study several game positions of $n-Sim$ and prove they are PII wins. We shall always denote PI's moves in green and PII's in red.

We first define the following configuration on 5 vertices, and call it a 'drawn $K_5$':

\begin{center}
    \includegraphics[scale=0.6]{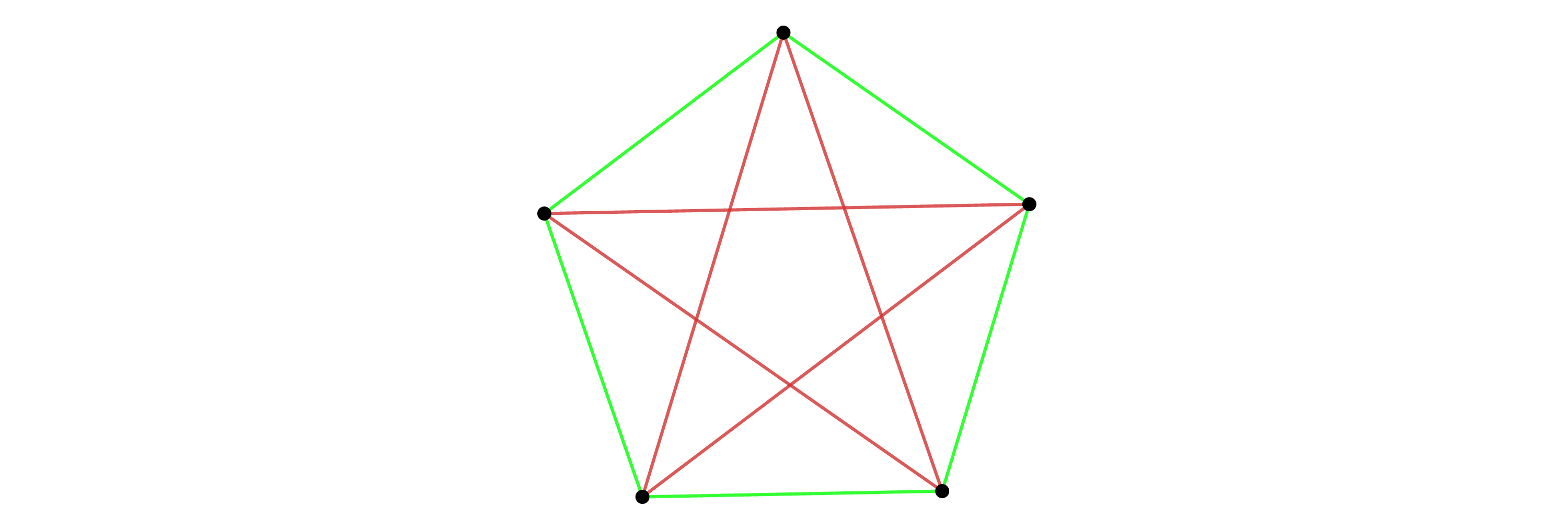}

    A drawn $K_5$
\end{center}

The existence of the drawn $K_5$ shows that $5-Sim$ can end in a draw. In fact, the drawn $K_5$ is the unique draw position up to isomorphism.

It is therefore natural to ask what happens if, say for $n$ a multiple of 5, the players start off by playing so as to decompose the board into disjoint drawn $K_5$'s. This is our first result:

\begin{theorem} For $n$ a multiple of 5, the position with the board decomposed into $n/5$ disjoint drawn $K_5$'s is a second-player win.
\end{theorem}

This might lead us to believe that, for any value of $n\geq6$, a configuration consisting of only a drawn $K_5$ might be a PII win. However, this is $not$ true, as we will show in the case $n=7$. So the following result is rather surprising: faced with a drawn $K_5$ minus one edge, although completing the  drawn $K_5$ may be losing for PII, he does have a winning move.

\begin{theorem} For $n\geq6$, a position of $n-Sim$ consisting of only a drawn $K_5$ minus one red edge is a PII win.
\end{theorem}

In fact, it turns out that we can push back our analysis of $n-Sim$ to quite close to the start of the game:

\begin{theorem} For $n\geq6$, the following position of $n-Sim$ consisting of only $3$ moves is a PII win:

\begin{center}
    \includegraphics[scale=0.6]{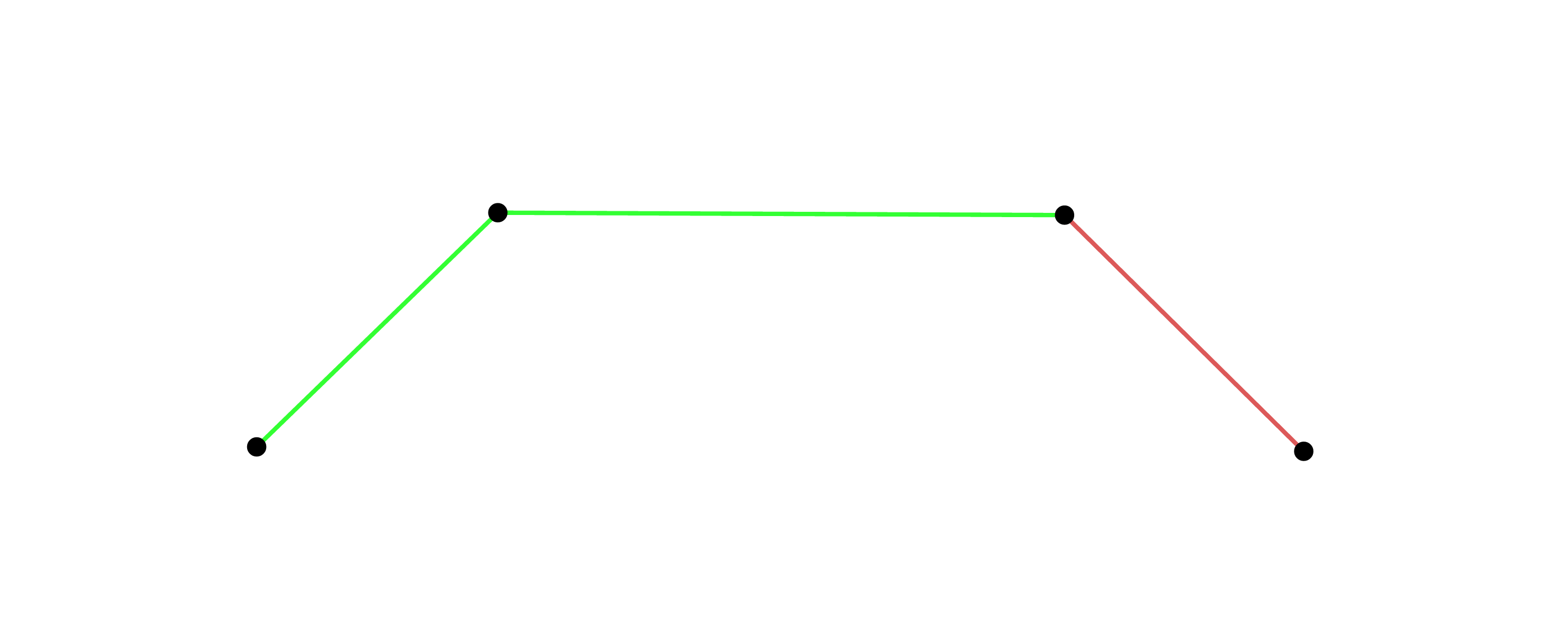}
\end{center}
\end{theorem}
\newpage
The proofs of all three results involve notions of strategy stealing. Our key idea to make the argument work in the mis\`ere setting is that, given a configuration S with one more green edge than red and supposing it is a PI win, PII ignores one of the green edges and pretends there is an extra red one so that the new configuration is colour-swapped-isomorphic to the original one. {\bf Crucially}, we need some 'insurance' on one of the two edges that PII is pretending about, e.g. the green edge that we are ignoring forms a triangle with two other red edges, so PII will never be called upon to play the preexisting green edge while strategy stealing.

\vskip 30pt

\section*{\normalsize 2. Discussion of drawn $K_5$ configurations}

We first prove Theorem 1.
\begin{proof}
Suppose that $n=5k, k\geq2$ and the two players have partitioned the board into $k$ disjoint drawn $K_5$'s. Call this configuration $S$. Suppose that PI has a winning strategy from $S$. Notice that from this configuration, PI only has one move available, up to isomorphism, and wlog the game state after this move is the following (where there are in total $k$ drawn $K_5$'s):
\begin{center}
    \includegraphics[scale=0.35]{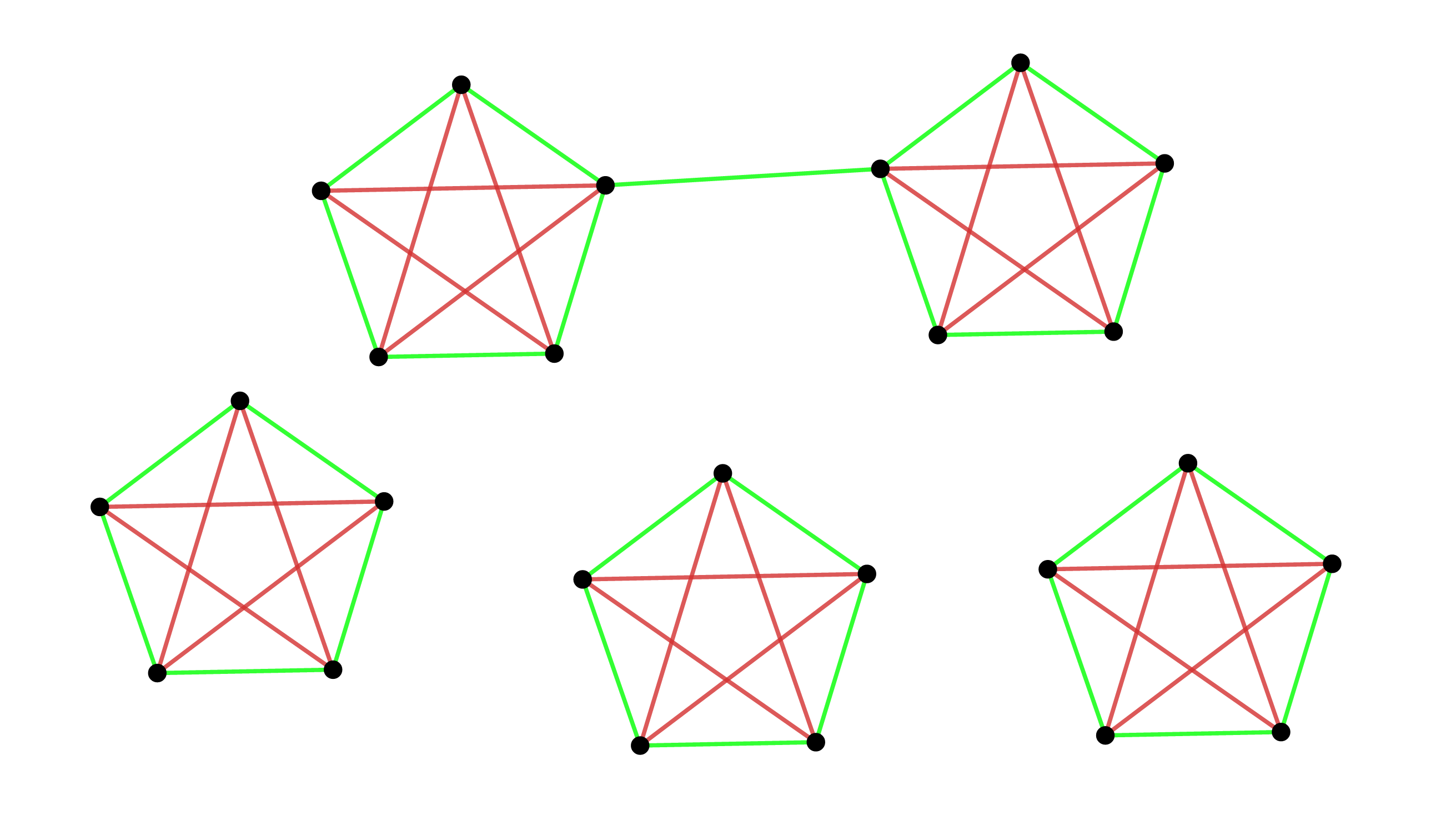}
\end{center}

But we have supposed that PI has a winning strategy which requires him to start by drawing any green edge joining two vertices of distinct copies of $K_5$. Now PII can pretend that this new green edge has not been drawn and the game state is still $S$. Notice that $S$ is isomorphic to $S'$ - that is, $S$ with the colours of the edges swapped. Therefore PII can use PI's presumed winning strategy, and start by drawing the following red edge (the dotted green edge is the one that PII is ignoring):

\begin{center}
    \includegraphics[scale=0.35]{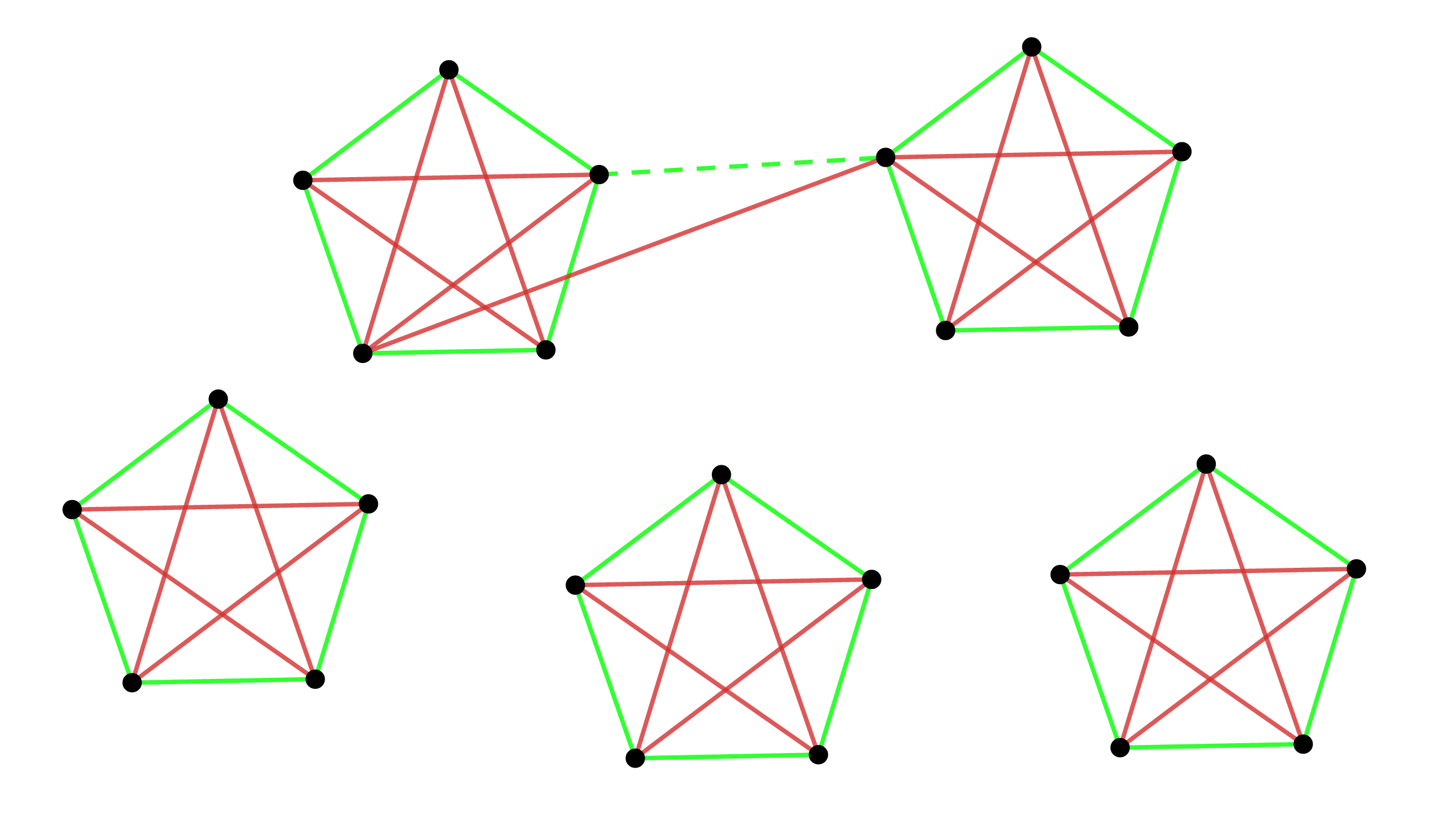}
\end{center}

From here onward PII continues to play according to the supposed PI winning strategy. Finally, notice that PII will never be required to colour the dotted green edge that he is ignoring, as that would close a red triangle, in contradiction with the fact that PII is following a winning strategy.
Therefore PI cannot have a winning strategy from configuration $S$, and as no draws are possible in $n-Sim$ for $n\geq6$, it follows that there exists a PII winning strategy from this game state.
\end{proof}

We remark that Theorem 1 also follows from results in [2]. We gave a proof above because we will build on this in later results.

We now turn our attention to the case where the initial configuration $T$ of $n-Sim$ consists of just one drawn $K_5$ (and $n-5$ isolated vertices). We have the following result:

\begin{proposition}

\textbf{}{(i)} \normalfont For any $n\geq6$, if PI, playing from position $T$ joins a vertex of the drawn $K_5$ to an isolated vertex, he loses. In particular, for $n=6$, PI always loses from position $T$.

\textbf{(ii)} \normalfont PI has a winning strategy starting from $T$ in $7-Sim$.
\end{proposition}
\begin{proof}
\textbf{(i)} Suppose PI has a winning strategy from position $T$ requiring him to connect an isolated point to a vertex of the drawn $K_5$, i.e. suppose the following game state is a PI win:

\begin{center}
    \includegraphics[scale=0.4]{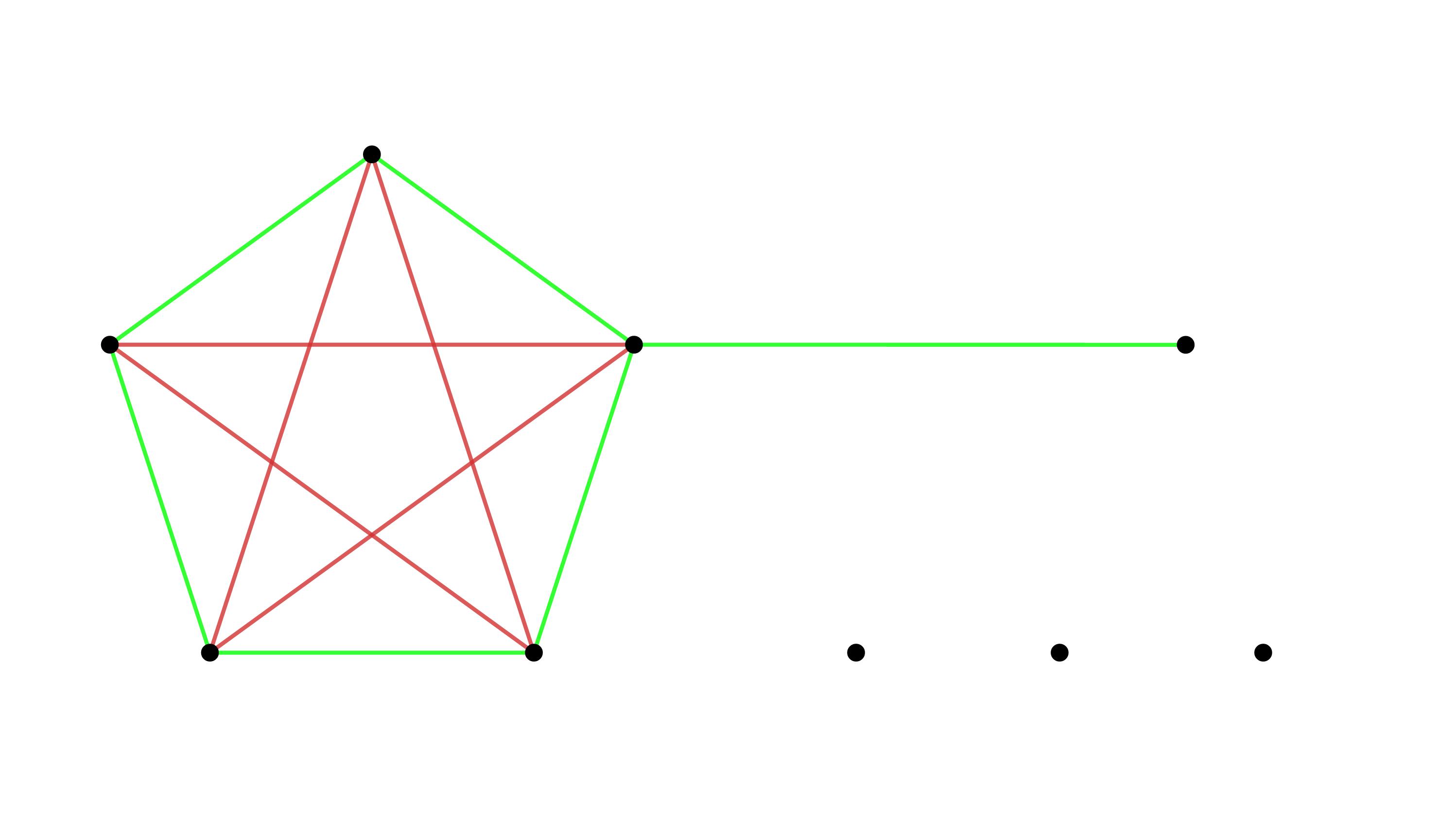}
\end{center}

Then PII can ignore the last move PI made and draw the following edge:

\begin{center}
    \includegraphics[scale=0.4]{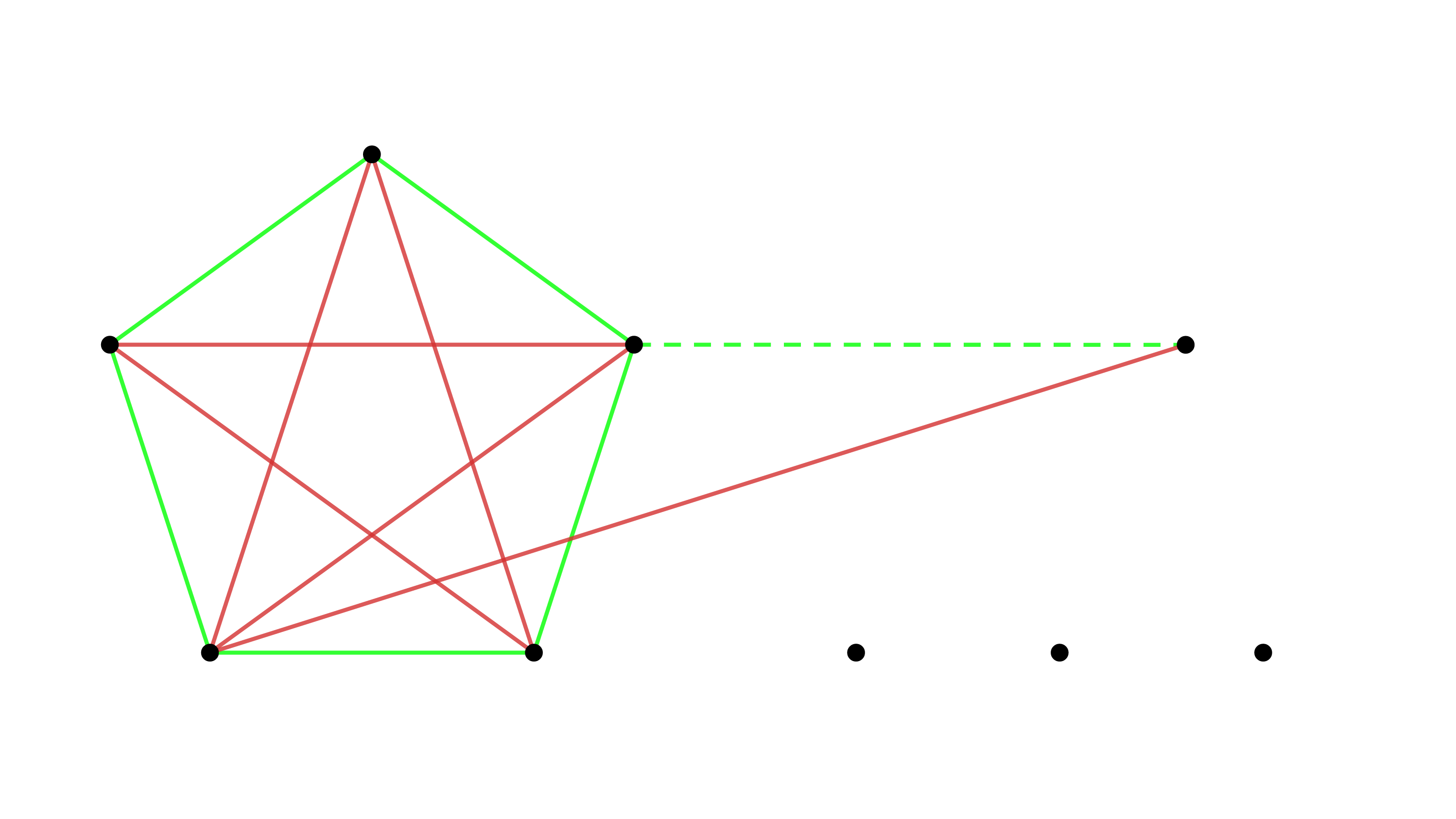}
\end{center}

Then the position that PII is in is isomorphic to the colour-inverted previous position, therefore PII can follow the presumed winning strategy of PI. Finally, notice that PII will never be called upon to play the dashed green edge that he is ignoring, since that would close a red triangle. Therefore PII can steal PI's presumed winning strategy.
In particular, this means that PII always wins in $6-Sim$ starting from a drawn $K_5$ and one isolated vertex.

\textbf{(ii)} We proceed by exhaustion. By \textbf{(i)}, we know that PI must start by connecting the two isolated vertices. Call them $X$ and $Y$, and label the 5 vertices of the drawn $K_5$ from A to E.

\begin{center}
    \includegraphics[scale=0.35]{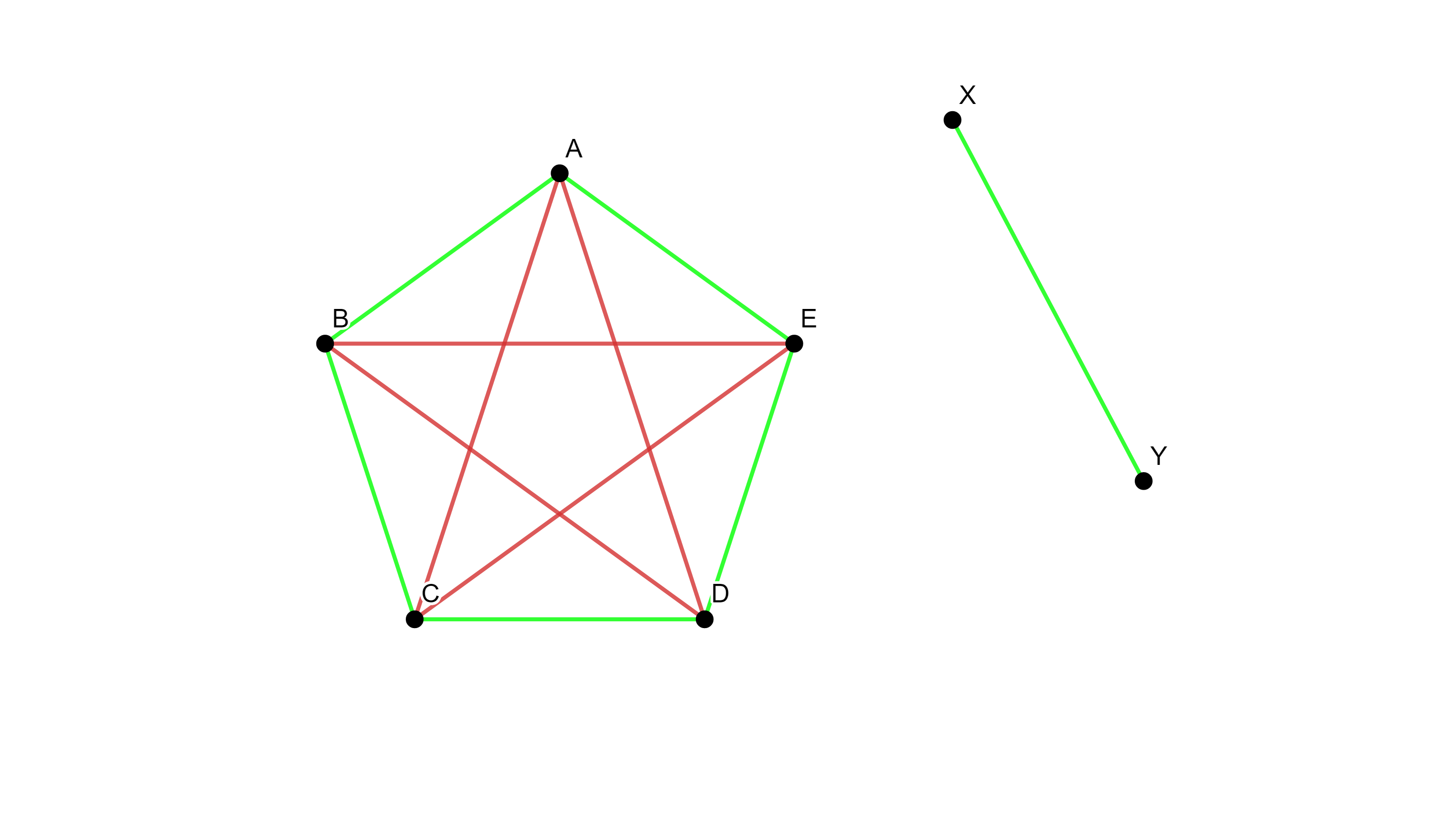}
\end{center}

Note that each player can draw at most two edges from either of $X$ or $Y$ into the drawn $K_5$, meaning 4 moves in total. Since PII goes first from this configuration, to ensure a PI win it is sufficient to show that he can always make at least as many moves as PII. Now PII has a unique move available, up to isomorphism, so wlog he draws the edge connecting $X$ and $A$. Then PI draws $XB$, and the game state is as follows, with $XE$ shown in dotted red being the only edge that PII can still draw from $X$:

\begin{center}
    \includegraphics[scale=0.35]{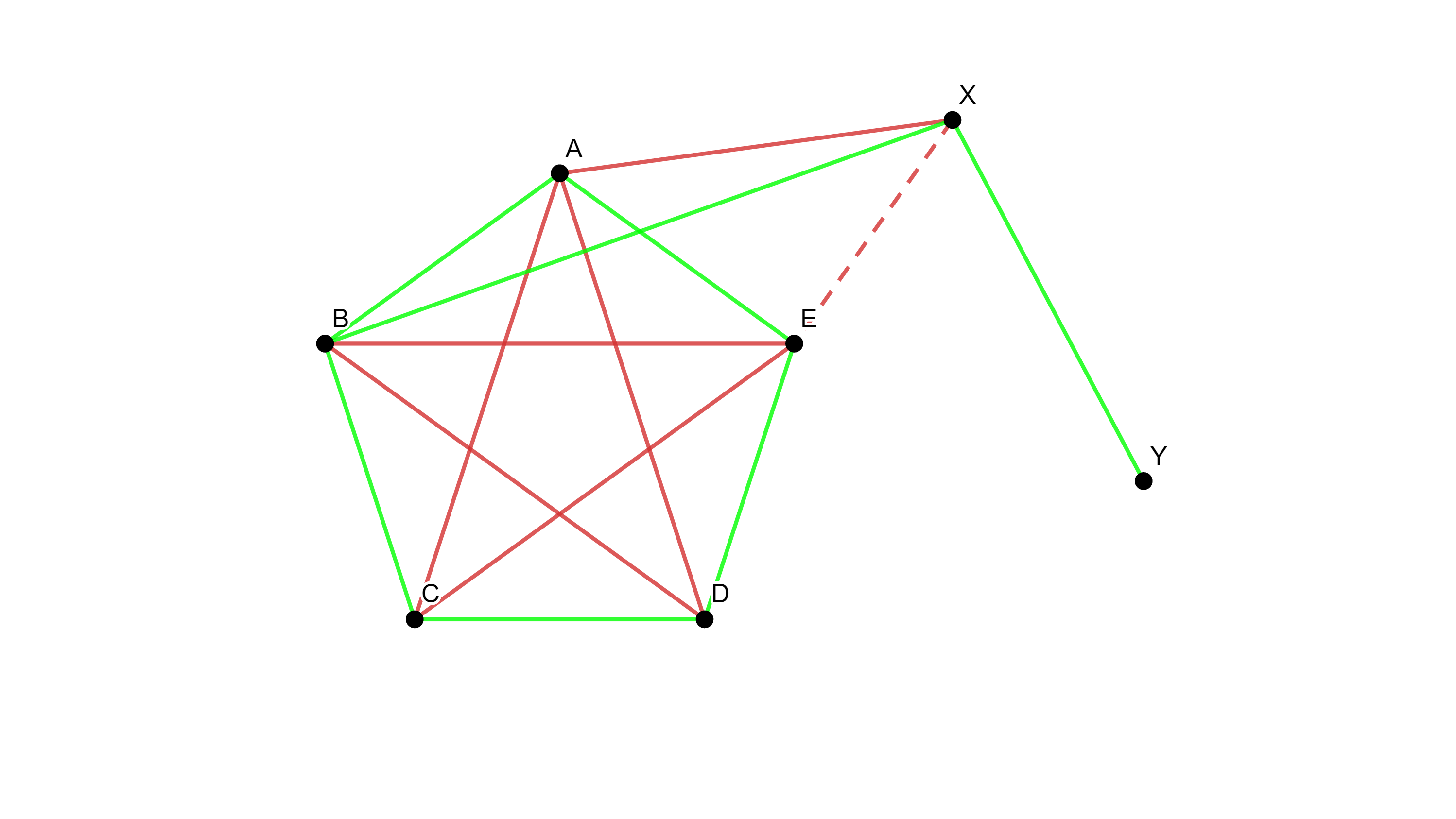}
\end{center}

\textbf{Case 1} PII draws the dotted red edge $XE$ on his next move. Then PI draws $YC$. But now, regardless of what PII's next move is, PI can draw at least one of $YA$ and $YE$ without closing a green triangle, and after PII's fourth move, PI draws $XD$. Therefore PI wins.

\textbf{Case 2} PII draws an edge other than $XE$, i.e. an edge connecting $Y$ and a vertex of the drawn $K_5$. Then PI draws $XE$, and therefore PII has no more moves available that touch $X$, hence can only make one more move. Then regardless of PII's third (and final) move, at least one of $YA$, $YC$ and $YD$ will be unoccupied, so PI can draw one of them without closing a green triangle, therefore again winning.
\end{proof}

Part $(ii)$ of the previous Proposition is in stark contrast to both Theorems 1 and 2. Theorem 2 tells us that a drawn $K_5$ missing one edge is always a PII win, but PII's drawing of this missing edge is not necessarily a winning move. We now prove this.
\begin{proof}
Consider the missing edge $e$ of the drawn $K_5$. If this is a winning move for PII, then he draws it. Suppose now that $e$ is a losing move for PII. This implies that a drawn $K_5$ is a PI winning position. But as the drawn $K_5$ is isomorphic to itself with the colours swapped, PII can pretend that he's already drawn $e$ and now employs the winning strategy starting from the drawn $K_5$. Notice that PI can never draw $e$ as it would close a green triangle.
Therefore for any $n\geq6$, PII has a winning strategy starting from a drawn $K_5$ minus one (red) edge.
\end{proof}

\vskip 30pt

\section*{\normalsize 3. Proof of Theorem 3}

Theorem 3 follows easily using the same strategy-stealing arguments as before.

\begin{proof}

For convenience, we label the edges of the configuration (call it $U$) as follows:
\begin{center}
    \includegraphics[scale=0.6]{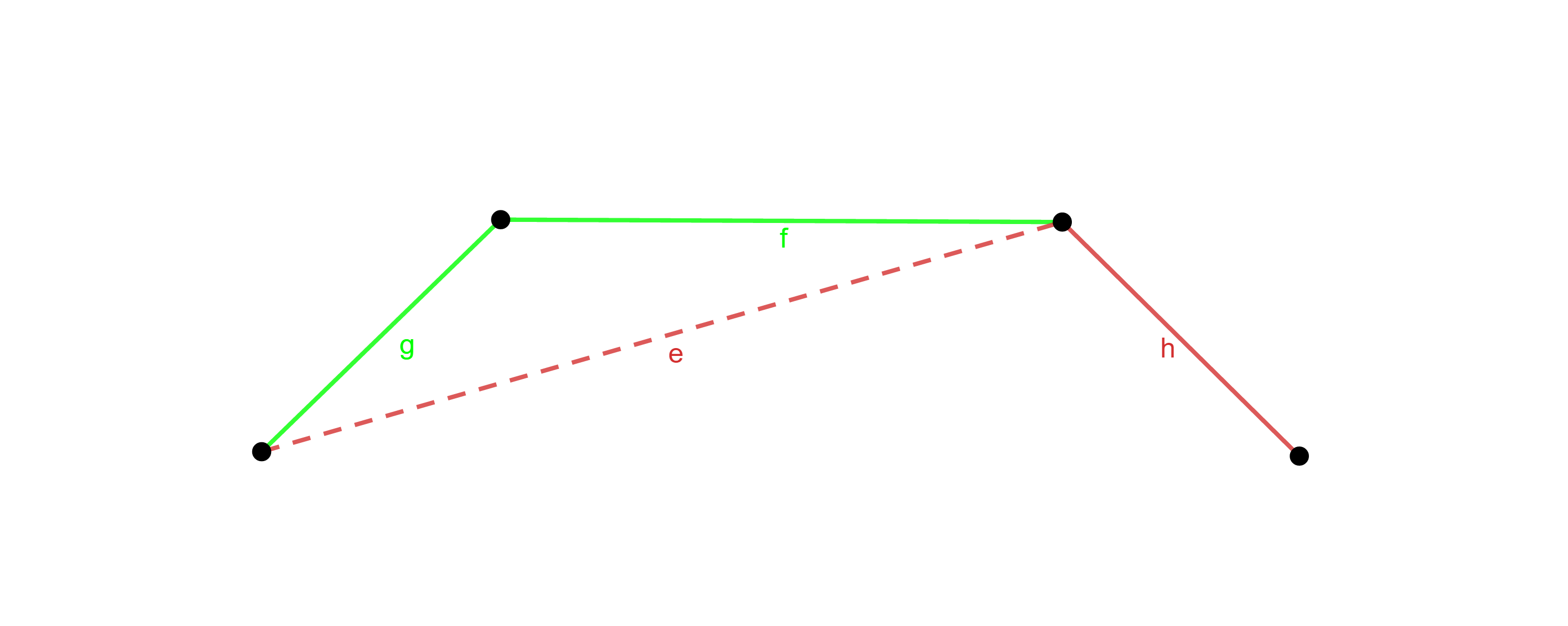}
\end{center}
We have also included a dotted edge $e$ that is not part of $U$.

Suppose that configuration $U$ is a PI win. Then PII pretends that he is the first player of the game and that the game position consists of the edges $e,h$ in red and $g$ in green, but not $f$, and it is his opponent's turn. Then the configuration is a colour-swapped $U$, therefore PII has a winning strategy by our assumption. But then PII can further pretend that his opponent's next move was colouring edge $f$, therefore PII can play according to the presumed winning strategy. Contradiction.

Crucially, in the above argument PII pretends that edge $e$ is already red but never colours it, and PI can never color it as that would close a green triangle.
\end{proof}

Unfortunately, we do not see any way to push the argument of Theorem 3 back to the starting position (or, rather, the position after PI's first move). Some new ideas would be needed for this.

\section*{\normalsize Acknowledgements}
The author wishes to thank Professor Imre Leader for his helpful insights and suggestions during the development of this article.

\section*{\normalsize References}

\hspace{\parindent}[1] J. Beck, {\it Combinatorial Games, Volume 114 of Encyclopedia of Mathematics and its Applications}, Cambridge University Press, Cambridge, 2008.

[2] J. R. Johnson, I. Leader, and M. Walters, Transitive avoidance games, {\it Electron. J. Combin.}  {\bf24} (2017), P1.61

[3] G. Simmons, The game of SIM, {\it J. Rec. Math.} (1969), 2(2):66.

[4] W. Slany, Graph Ramsey games, {\it Electronic Colloquium on Computational Complexity (ECCC)} {\bf47} (1999).

\end{document}